\pdfoutput=1
\documentclass{amsart}

\usepackage[mathscr]{eucal}
\usepackage{amssymb}
\usepackage[dvipsnames]{xcolor} 
\usepackage[normalem]{ulem}
\usepackage{wasysym}
\usepackage{amsthm}
\usepackage{bbold}
\usepackage{comment}
\usepackage{enumitem}
\usepackage{amsmath}
\usepackage{tikz-cd}
\usepackage{stmaryrd} 
\usepackage{etoolbox} 
\usepackage{microtype}
\usepackage{adjustbox}
\usepackage{ mathdots }
\usepackage{bbm}
\usepackage{mathtools}

\usepackage[unicode]{hyperref} 

\definecolor{dark-red}{rgb}{0.5,0.15,0.15}
\definecolor{dark-blue}{rgb}{0.15,0.15,0.6}
\definecolor{dark-green}{rgb}{0.15,0.6,0.15}

\hypersetup{
    colorlinks, linkcolor=OliveGreen,
    citecolor=OliveGreen, urlcolor=OliveGreen
}

\usepackage[nameinlink,capitalise,noabbrev]{cleveref}

\usepackage[all]{xy}
\xyoption{line}
\usepackage{graphicx}
\usepackage{caption}

\newtheorem{thmx}{Theorem}

\newtheorem{Thm}[equation]{Theorem}
\newtheorem*{Thm*}{Theorem}
\newtheorem*{MainThm*}{Main Theorem}
\newtheorem{Prop}[equation]{Proposition}
\newtheorem{Lem}[equation]{Lemma}
\newtheorem{Cor}[equation]{Corollary}
\newtheorem*{Cor*}{Corollary}

\newtheorem*{Que*}{Question}

\theoremstyle{remark}
\newtheorem{Def}[equation]{Definition}

\newtheorem{Exa}[equation]{Example}

\newtheorem{Rem}[equation]{Remark}

\tikzset{
    labelrotatebelow/.style={anchor=north, rotate=90, inner sep=1.0mm}
}
\tikzset{
    labelrotateabove/.style={anchor=south, rotate=90, inner sep=1.0mm}
}
\SelectTips{cm}{10}

\newcommand{\nc}{\newcommand}
\nc{\dmo}{\DeclareMathOperator}

\usepackage{todonotes}

\nc{\overbar}[1]{\mkern 1.5mu\overline{\mkern-1.5mu#1\mkern-1.5mu}\mkern 1.5mu}

\nc{\kappaaux}{g}
\nc{\kappaCh}{{\kappaaux(\cat C_h)}}
\nc{\kappam}{{\kappaaux({\mathfrak m})}}
\nc{\kappaP}{\Gamma_{\cat P}\unit}
\nc{\kappaQ}{{\kappaaux(\cat Q)}}
\nc{\kappaCP}{{\kappaaux_{\cat C}(\cat P)}}
\nc{\kappaDP}{{\kappaaux_{\cat D}(\cat P)}}
\nc{\kappaCQ}{{\kappaaux_{\cat C}(\cat Q)}}
\nc{\kappaDQ}{{\kappaaux_{\cat D}(\cat Q)}}
\nc{\kappaphiB}{{\kappaaux(\phi(\cat B))}}
\nc{\kappaphiQ}{{\kappaaux(\varphi(\cat Q))}}

\dmo{\Sub}{Sub}
\dmo{\Proj}{Proj}
\dmo{\LMod}{LMod}
\dmo{\cell}{cell}
\nc{\Prst}{{\cat P}\mathrm{r^{st}}}
\nc{\Mack}[2]{\mathrm{Mack}_{#1}(#2)}
\dmo{\fin}{{fin}}
\dmo{\Sphere}{\mathbb{S}}
\dmo{\Alg}{Alg}
\dmo{\CAlg}{CAlg}
\nc{\HA}{{\rmH \hspace{-0.2em}\bbA}}
\nc{\HZ}{{\rmH \hspace{-0.2em}\bbZ}}
\nc{\HZbar}{{\rmH \hspace{-0.2em}\underline{\bbZ}}}
\nc{\Fp}{{\bbF_{\hspace{-0.1em}p}}}
\nc{\HFp}{{\rmH \hspace{-0.15em}\bbF_{\hspace{-0.1em}p}}}

\nc{\mathfrakp}{\mathfrak{p}}
\nc{\mathfrakq}{\mathfrak{q}}
\nc{\mathfrakS}{\mathfrak{S}}
\nc{\mathfrakT}{\mathfrak{T}}
\nc{\Z}{\mathbb{Z}}
\nc{\cF}{\mathcal{F}}
\nc{\hspec}[1]{\Spc^\mathrm{h}({#1})}

\dmo{\Id}{Id}
\dmo{\Loc}{Loc}
\dmo{\Spc}{Spc}
\nc{\thickid}{\mathrm{thickid}}
\nc{\thick}[1]{\mathrm{thick}\langle #1 \rangle}
\nc{\thickt}[1]{\langle #1 \rangle_\otimes}
\nc{\loct}[1]{\mathrm{Loc}_\otimes\langle #1 \rangle}
\dmo{\End}{End}
\dmo{\Mor}{Mor}
\dmo{\Hom}{Hom}
\dmo{\id}{id}
\dmo{\im}{im}
\dmo{\Ker}{Ker}
\dmo{\ind}{ind}
\dmo{\Ind}{Ind}
\dmo{\CoInd}{coind}
\dmo{\res}{res}
\dmo{\infl}{infl}
\dmo{\triv}{triv}
\dmo{\Tel}{Tel} 
\dmo{\grMod}{grMod}%
\dmo{\Mod}{Mod}%
\dmo{\opname}{op}
\dmo{\SH}{\mathcal{S}\mathcal{H}}
\dmo{\Spec}{Spec}
\dmo{\supp}{supp}
\dmo{\Supp}{Supp}
\dmo{\cosupp}{cosupp}
\dmo{\Cosupp}{Cosupp}

\nc{\bbL}{\mathbb{L}}
\nc{\bbA}{\mathbb{A}}
\nc{\bbE}{\mathbb{E}}
\nc{\bbN}{\mathbb{N}}
\nc{\bbQ}{\mathbb{Q}}
\nc{\bbZ}{\mathbb{Z}}
\nc{\bbF}{\mathbb{F}}
\nc{\bbS}{\mathbb{S}}
\nc{\cat}[1]{\mathscr{#1}}

\nc{\cA}{\mathcal{A}}
\nc{\cB}{\mathcal{B}}
\nc{\cC}{\mathcal{C}}
\nc{\cD}{\mathcal{D}}
\nc{\cE}{\mathcal{E}}
\nc{\cU}{\mathcal{U}}

\nc{\CB}{\mathbf{B}}

\nc{\sD}{\mathsf{D}}

\nc{\Mid}{\,\big|\,}
\nc{\SET}[2]{\big\{\,#1\Mid#2\,\big\}}
\nc{\unit}{\mathbb{1}}
\nc{\xra}{\xrightarrow}

\dmo{\Sp}{Sp}
\dmo{\Ho}{Ho}
\dmo{\Fin}{Fin}
\dmo{\add}{add}
\dmo{\Fun}{Fun}
\dmo{\Ext}{Ext}

\dmo{\Map}{Map}
\dmo{\Span}{Span}
\dmo{\N}{N}
\dmo{\Cat}{Cat}
\dmo{\colim}{colim}
\dmo{\hocolim}{hocolim}
\dmo{\Ch}{Ch}
\dmo{\Gr}{Gr}
\nc{\CA}{\cat A}
\nc{\CU}{\cat U}

\dmo{\dual}{dual}
\dmo{\Perf}{Perf}
\dmo{\rk}{rk}
\dmo{\cons}{cons}

\renewcommand{\leq}{\leqslant}

\DeclareMathOperator{\inv}{inv}

\DeclareMathOperator{\gen}{gen}

\newcounter{enum-resume-hack}
\Crefname{Thm}{Theorem}{Theorems}
\Crefname{Prop}{Proposition}{Propositions}
\Crefname{Lem}{Lemma}{Lemmas}
\Crefname{thmx}{Theorem}{Theorems}

\begin{document}

\title[Cohen's theorem in tt-geometry]{Cohen's theorem in tensor triangular geometry}

\author[Tobias Barthel]{Tobias Barthel}

\date{\today}

\makeatletter
\patchcmd{\@setaddresses}{\indent}{\noindent}{}{}
\patchcmd{\@setaddresses}{\indent}{\noindent}{}{}
\patchcmd{\@setaddresses}{\indent}{\noindent}{}{}
\patchcmd{\@setaddresses}{\indent}{\noindent}{}{}
\makeatother

\address{Tobias Barthel, Max Planck Institute for Mathematics, Vivatsgasse 7, 53111 Bonn, Germany}
\email{tbarthel@mpim-bonn.mpg.de}
\urladdr{https://sites.google.com/view/tobiasbarthel/}

\begin{abstract} 
    A theorem of Cohen from 1950 states that a commutative ring is Noetherian if and only if every prime ideal is finitely generated. In this note, we establish analogues of this result in tensor triangular geometry. In particular, for an essentially small tensor triangulated category $\mathscr{K}$ with weakly Noetherian spectrum, we show that every prime ideal in $\mathscr{K}$ can be generated by finitely many objects if and only if the set of prime ideals of $\mathscr{K}$ is finite. 
\end{abstract}


\maketitle

\setcounter{tocdepth}{1}

\subsection*{Introduction} Let $p$ be a prime number and let $\Sp_{(p)}^{\fin}$ be the stable homotopy category of finite $p$-local spectra. The seminal thick subcategory theorem \cite{HopkinsSmith1998} classifies all thick subcategories of $\Sp_{(p)}^{\fin}$: they form a strictly decreasing sequence
    \[
        \Sp_{(p)}^{\fin} = \cat C_0 \supset \cat C_1 \supset \cat C_2 \supset \ldots \supset \cat C_{\infty} = (0),
    \]
with one $\cat C_n$ for any $n\in \mathbb{N} \cup \{\infty\}$. In applications, it is then of paramount importance that each of the thick subcategories $\cat C_n$ can be generated by a single object, namely a ($p$-local) finite spectrum of type $n$. As such, the thick subcategory theorem gives a classification of finite spectra in terms of type and up to the triangular operations of direct sum, retract, and taking cofibers. 

Tensor triangular geometry \cite{Balmer2005} provides a uniform framework for studying analogous classification problems for any tensor triangulated category $\cat K$, with prominent examples appearing in algebraic geometry, (equivariant stable) homotopy theory, representation theory, motivic homotopy theory, or symplectic geometry. The spectrum $\Spc(\cat K)$ of $\cat K$ encodes the classification of radical thick tensor ideals of $\cat K$. Motivated by the example of $\Sp_{(p)}^{\fin}$, this note studies the question:

\begin{center}
    \emph{When are all thick tensor ideals of a tt-category $\cat K$ finitely generated?}
\end{center}
    
Surprisingly, it turns out that this is rarely the case. In fact, under a mild point-set topological condition on the spectrum of $\cat K$ which holds for many examples commonly studied in tt-geometry, this already forces the set of all prime ideals in $\cat K$ to be finite. More precisely, recall that in \cite{BHS2023} we introduced the notion of weakly Noetherian spectral space as a common generalization of Noetherian and profinite. This notion is sufficient for setting up a theory of stratification for rigidly-compactly generated tt-categories $\cat T$ relative to their spectrum $\Spc(\cat T^{\omega})$, while more recently, Zou \cite{Zou2023pp} showed that it is also necessary. It captures a large collection of tt-categories that have recently been the focus of research. With that preparation, we can state a version of the main result of this note:

\begin{thmx}\label{thmx:main}
    Suppose $\cat K$ is an essentially small tt-category with weakly Noetherian spectrum. 
    Then the following are equivalent:
        \begin{enumerate}
            \item Every radical ideal in $\cat K$ is finitely generated;
            \item $\Spc(\cat K)$ is finite.
        \end{enumerate}
\end{thmx}

As the example of $p$-local finite spectra demonstrates, satisfying Condition $(a)$ but not $(b)$, this equivalence can fail without the weakly Noetherian hypothesis. Note that, in most examples of tt-categories arising in nature, all ideals are radical; for a sharper formulation of \cref{thmx:main}, see \cref{thm:main}. The proof of this result relies on Hochster duality to translate Condition $(a)$ into the statement that $\Spc(\cat K)$ is inverse-Noetherian; this does not require any assumption on the spectrum and forms a tt-geometric version of Cohen's theorem (\cref{prop:fgttcharacterization}):

\begin{thmx}\label{thmx:ttcohen}
    For a tt-category $\cat K$, the following are equivalent:
        \begin{enumerate}
            \item Every radical ideal in $\cat K$ is finitely generated;
            \item every prime ideal in $\cat K$ is finitely generated;
            \item $\Spc(\cat K)$ is inverse-Noetherian.
        \end{enumerate}
\end{thmx}

A novel point-set topological characterization of finite spectral spaces as those that are both weakly Noetherian and inverse-Noetherian then yields \cref{thmx:main}. 

When combined with the aforementioned theory of stratification for `big' tt-categories, we obtain as an immediate consequence that for tt-stratified rigidly-compactly generated tt-categories $\cat T$ with infinite spectrum there always exists an ideal of compacts in $\cat T$ which cannot be generated by finitely many objects. For instance, this applies to derived categories of Noetherian rings or schemes, various categories of representations \cite{BIK2011,BalmerGallauer2022pp,BBIKP2023pp,BCHNP2025}, or categories of homotopical or motivic origin, see e.g., \cite{BHS2023,BCHS2023pp, BalmerGallauer2025}. Further cases are studied throughout this note,  concluding with a novel example from global representation theory (\cref{ex:globalreptheory}) in which all ideals are finitely generated.

\subsection*{Topological preliminaries} We assume that the reader is familiar with the basic theory of spectral spaces including Hochster duality \cite{Hochster}, as for example summarized in \cite[Chapter 1]{DST2019}. For a spectral space $X$, the associated inverse (Hochster dual) space is denoted by $X_{\inv}$ and we write $X_{\cons}$ for the patch space of $X$, i.e., $X$ equipped with the corresponding constructible topology. If $\mathcal{P}$ is a property of spectral spaces, then we say that $X$ has \emph{inverse-$\mathcal{P}$} whenever $X_{\inv}$ satisfies $\mathcal{P}$.

By definition, a subset $W$ of a spectral space $X$ is \emph{Thomason} if it can be written as a union of complements of quasi-compact opens in $X$. We recall the following notions from \cite{BHS2023}:

\begin{Def}\label{def:wv+wN}
    Let $X$ be a spectral space. A subset $V$ in $X$ is called \emph{weakly visible} if there exists Thomason subsets $W_1,W_2$ in $X$ such that $V = W_1 \cap (X\setminus W_2)$; a point in $X$ is \emph{weakly visible} if the corresponding singleton is so. If every point of $X$ is weakly visible, we say $X$ is \emph{weakly Noetherian}.
\end{Def}

    Any spectral space which is Noetherian or Hausdorff is weakly Noetherian. The next result characterizes this property in terms of the inverse topology:

\begin{Lem}\label{lem:wvinverse}
    A subset of a spectral space $X$ is weakly visible if and only if it is locally closed in $X_{\inv}$.
\end{Lem}
\begin{proof}
    This follows from the fact that a subset of $X$ is Thomason if and only if it is inverse-open. Indeed, a subset is Thomason if and only if it is a union of complements of quasi-compact opens. Now quasi-compact opens form a basis of closed subsets for $X_{\inv}$, which verifies the claim.
\end{proof}

\begin{Lem}\label{lem:finitenesscharacterization}
    A spectral space $X$ is finite if and only if it is weakly Noetherian and inverse-Noetherian.
\end{Lem}
\begin{proof}
    By \cref{lem:wvinverse}, $X$ is weakly Noetherian if and only if every point in $X_{\inv}$ is locally closed. The latter condition is equivalent to every point of $X_{\cons} = (X_{\inv})_{\cons}$ being isolated provided that $X_{\inv}$ is Noetherian, see \cite[Corollary 8.1.19(i)]{DST2019}. In other words, 
    $X$ being both weakly Noetherian and inverse-Noetherian implies that $X_{\cons}$ is discrete. But $X_{\cons}$ is compact, hence it must be finite, so $X$ is finite as well. Conversely, any finite spectral space is both Noetherian and inverse-Noetherian.
\end{proof}

\Cref{lem:finitenesscharacterization} strengthens one of the equivalences of \cite[Corollary 8.1.17]{DST2019}, by relaxing the assumption on the space from Noetherian to weakly Noetherian.

\subsection*{The main result} Throughout, we will freely use the language of tt-geometry, as originally developed by Balmer in \cite{Balmer2005}.

\begin{Def}\label{def:fgttideal}
    A radical ideal $\cat I$ in a tt-category $\cat K$ is said to be \emph{finitely generated} if there exists a finite collection of objects $X_1,\ldots, X_n \in \cat K$ such that $\cat I = \sqrt{\thickt{X_1,\ldots,X_n}}$. By passing to the direct sum of the $X_i$'s, we may equivalently require $\cat I$ to be \emph{principal}, i.e., generated by a single object in $\cat K$.
\end{Def}

The next proposition is our tt-geometric incarnation of Cohen's theorem \cite[Theorem 2]{Cohen1950}. We begin with an auxiliary result, probably well-known to the experts:

\begin{Lem}\label{lem:ttfg}
    A radical ideal $\cat I$ in an essentially small tt-category $\cat K$ is finitely generated if and only if the complement $\supp(\cat I)^c$ of its support in $\Spc(\cat K)$ is constructible. 
\end{Lem}
\begin{proof}
    As remarked before, $\cat I$ is finitely generated if and only if there exists some $x \in \cat K$ such that $\cat I = \sqrt{\langle x \rangle_{\otimes}}$. By the classification theorem \cite[Theorem 4.10]{Balmer2005}, this equality in turn is equivalent to $\supp(\cat I) = \supp(x)$. Subsets of $\Spc(\cat K)$ of the form $\supp(x)$ for some $x \in \cat K$ are characterized by the property that their complements are quasi-compact open, see \cite[Proposition 2.14]{Balmer2005}. By \cite[Theorem 1.5.4(iii)]{DST2019}, a subset of a spectral space is quasi-compact open if and only if it is constructible and closed under generalization. Since any subset of the form $\supp(\cat I)^c$ is automatically closed under generalization, the claim follows.
\end{proof}

\begin{Prop}\label{prop:fgttcharacterization}
    For a tt-category $\cat K$, the following are equivalent:
        \begin{enumerate}
            \item $\Spc(\cat K)$ is inverse-Noetherian;
            \item every radical ideal in $\cat K$ is finitely generated;
            \item every prime ideal in $\cat K$ is finitely generated.
        \end{enumerate}
\end{Prop}
\begin{proof}
    A subset of $\Spc(\cat K)$ is the support of a radical ideal in $\cat K$ if and only if it is a Thomason subset, i.e., if and only if it is inverse-open. In other words, a subset of $\Spc(\cat K)$ is of the form $\Supp(\cat I)^c$ for some radical ideal $\cat I \subseteq \cat K$ if and only if it is inverse-closed. \Cref{lem:ttfg} therefore shows that all radical ideals in $\cat K$ are finitely generated if and only if all inverse-closed subsets of $\Spc(\cat K)$ are constructible. Similarly, for a prime ideal $\cat P$, we have $\supp(\cat P)^c = \gen(\cat P)$, the generalization closure of $\cat P$. This means that all prime ideals in $\cat K$ are finitely generated if and only if all subsets of the form $\gen(\cat P)$ in $\Spc(\cat K)$ are constructible, again using \cref{lem:ttfg}. 
    
    Applying \cite[Proposition 8.1.11]{DST2019} to $\Spc(\cat K)$, specifically the equivalence of (i), (iii), and (v) there, we see that the following conditions are equivalent:
        \begin{enumerate}
            \item[$(a)$] $\Spc(\cat K)$ is inverse-Noetherian;
            \item[$(b')$] every inverse-closed subset in $\Spc(\cat K)$ is constructible;
            \item[$(c')$] $\gen(\cat P)$ is constructible for each $\cat P \in \Spc(\cat K)$.
        \end{enumerate}
    The previous paragraph says that $(b)$ and $(b')$ are equivalent, as are $(c)$ and $(c')$. 
\end{proof}

With all the pieces in place, we can now assemble the proof of our main result:

\begin{Thm}\label{thm:main}
    Let $\cat K$ be an essentially small tt-category. Then the following conditions are equivalent:    
        \begin{enumerate}
            \item $\Spc(\cat K)$ is weakly Noetherian and every radical ideal in $\cat K$ is finitely generated;
            \item $\Spc(\cat K)$ is weakly Noetherian and every prime ideal in $\cat K$ is finitely generated;
            \item $\Spc(\cat K)$ is finite.
        \end{enumerate}
\end{Thm}
\begin{proof}
    \Cref{prop:fgttcharacterization} shows that each of $(a)$ or $(b)$ are equivalent to the statement that $\Spc(\cat K)$ is weakly Noetherian and inverse-Noetherian. This in turn is equivalent to $\Spc(\cat K)$ being finite by \cref{lem:finitenesscharacterization}, which is $(c)$.
\end{proof}

\begin{Rem}\label{rem:allideals}
    Observe that the spectrum of an essentially small tt-category $\cat K$ is finite if and only if the set of all radical ideals in $\cat K$ is finite. Indeed, radical ideals are in bijection with inverse-open subsets of $\Spc(\cat K)$. The claim then follows because the set of opens of a spectral space $X$ is finite if and only if $X$ is finite.  
\end{Rem}

\subsection*{Application to tt-stratification}
A tt-category $\cat T$ is said to be \emph{rigidly-compactly generated} if it is compactly generated by its full subcategory of compact objects $\cat T^{\omega}$ and if the compact objects in $\cat T$ coincide with the dualizable ones (i.e., $\cat T^{\omega}$ is \emph{rigid}). The next definition recalls the notion of tt-stratification introduced in \cite{BHS2023} under the assumption that the spectrum is weakly Noetherian. Subsequently, this was generalized by Zou \cite{Zou2023pp} to arbitrary rigidly-compactly generated tt-categories, building on earlier work of Sanders \cite{Sanders2017pp} extending Balmer--Favi support \cite{BalmerFavi2011}.  

\begin{Def}\label{def:ttstratification}
    Suppose $\cat T$ is a rigidly-compactly generated tt-category. We say that $\cat T$ is \emph{tt-stratified} if the Balmer--Favi--Sanders notion of support induces a bijection
        \[
            \Supp\colon \big\{ \text{localizing ideals of $\cat T$} \big\} \xra{\cong} \big\{ \text{subsets of $\Spc(\cat T^{\omega})$}\big\}.
        \]
\end{Def}

    As a consequence of our main theorem, tt-stratified categories must contain an ideal of compacts that cannot be finitely generated unless the spectrum is finite; examples abound, see e.g., \cite{BIK2011,BBIKP2023pp,BCHS2023pp,BCHNP2025,BalmerGallauer2025}.

\begin{Cor}\label{cor:ttstratifiedfg}
    Let $\cat T$ be a tt-stratified tt-category. If $\Spc(\cat T^{\omega})$ is infinite, then there exists a prime ideal in $\cat T^{\omega}$ which cannot be generated by finitely many objects.
\end{Cor}
\begin{proof}
    Zou \cite[Theorem 8.13]{Zou2023pp} proved that a tt-stratified rigidly-compactly generated tt-category must necessarily have a weakly Noetherian spectrum. If $\Spc(\cat T^{\omega})$ is infinite, then \cref{thm:main} implies that $\cat T^{\omega}$ contains a prime ideal that cannot be finitely generated.
\end{proof}

\subsection*{Examples}
Finally, we collect a number of examples that illustrate our main results, highlighting the key phenomena. The first example demonstrates a key difference between Cohen's theorem in commutative algebra and its tt-geometric counterpart.

\begin{Exa}\label{ex:commutativerings}
    Let $R$ be a commutative ring and write $\cat D(R)$ for its (unbounded) derived category. By the Hopkins--Neeman--Thomason theorem \cite{Neeman92,Thomason1997}, there is a natural homeomorphism
        \[
            \Spc(\cat D(R)^{\omega}) \cong \Spec(R).
        \]
    If $R$ is Noetherian, then the Zariski spectrum $\Spec(R)$ is Noetherian as a topological space, so \cref{thm:main} applies: every (prime) ideal in $\cat D(R)^{\omega}$ is finitely generated if and only if there are only finitely many prime ideals in $R$ (so $\dim(R) \leq 1$). This uses that $\cat D(R)$ is rigidly-compactly generated so that every ideal in $\cat D(R)^{\omega}$ is radical.

    This marks a sharp contrast between the tt-geometry of $\cat D(R)$ and the commutative algebra of $R$: By the aforementioned theorem of Cohen, a commutative ring in which all prime ideals are finitely generated is Noetherian, but this property cannot be detected in the Zariski spectrum. Indeed, there are many non-Noetherian rings whose spectrum is a point, such as $k[x_1,x_2,\ldots]/(x_1^2,x_2^2,\ldots)$ or variations thereof. However, we emphasize that the finite generation of all prime ideals in a given tt-category should \emph{not} be viewed as a tt-theoretic analogue of being Noetherian.
\end{Exa}

\begin{Exa}\label{ex:equivariantspectra}
    Let $G$ be a compact Lie group and consider the category $\Sp_G$ of $p$-local $G$-spectra, with the prime number $p$ being implicit. Geometric fixed points induce a jointly conservative family of geometric functors
        \[
            \Phi \coloneqq (\Phi^H \colon \Sp_G \to \Sp)_{(H)}
        \]
    indexed on conjugacy classes $(H)$ of subgroups of $G$. By \cite[Theorem 1.4]{BCHS2024}, this results in a surjective (in fact: bijective \cite[Theorem 3.14]{BGH2020}) map
        \[
            \varphi \coloneqq \Spc(\Phi)\colon \textstyle\bigsqcup_{(H)}\Spc(\Sp^{\omega}) \to \Spc(\Sp_G^{\omega}).
        \]
    When $G$ is finite, the source is a finite union of inverse-Noetherian spaces, hence itself inverse-Noetherian. It follows that $\Spc(\Sp_G^{\omega})$ is also inverse-Noetherian in this case. We deduce from \cref{prop:fgttcharacterization} that all ideals of $\Sp_G^{\omega}$ are finitely generated when $G$ is finite. In contrast, this fails for positive dimensional compact Lie groups, see \cite[Example 5.5]{BGH2020} for an explicit example for $G = U(1)$.

    More generally, rationalization furnishes an embedding
        \[
            \Spc(\Sp_G^{\omega}) \hookleftarrow \Spc(\Sp_{G,\bbQ}^{\omega}).
        \]
    The spectral space $\Spc(\Sp_{G,\bbQ}^{\omega})$ is generically Noetherian hence also weakly Noetherian by \cite[Theorem D]{BalchinBarthelGreenlees2023pp}, with underlying patch space $\Spc(\Sp_{G,\bbQ}^{\omega})_{\cons}$ the space of conjugacy classes of closed subgroups of $G$ with the Hausdorff metric topology (\cite[Theorem B]{BalchinBarthelGreenlees2023pp}). In particular, $\Spc(\Sp_{G,\bbQ}^{\omega})$ is finite if and only if $G$ is finite. Therefore, if $G$ is not a finite group, then $\Spc(\Sp_{G,\bbQ}^{\omega})$ is not inverse-Noetherian by \cref{lem:finitenesscharacterization}, so $\Spc(\Sp_G^{\omega})$ cannot be inverse-Noetherian either. In summary, we deduce abstractly from \cref{prop:fgttcharacterization} that the tt-category $\Sp_G^{\omega}$ contains an ideal that cannot be finitely generated if and only if $G$ is infinite.
 
    A similar analysis applies to the category of equivariant spectra for a profinite group, see \cite{BBB2024pp}, or excisive functors from spectra to spectra, see \cite{ABHS2024pp}.
\end{Exa}

We finish this note with an example that in fact provided the initial inspiration for studying when all radical ideals in a tt-category are finitely generated. 

\begin{Exa}\label{ex:globalreptheory}
    Write $\cat E_p$ for the family of finite elementary abelian $p$-groups for some prime number $p$. Let $\Sp_{\cat E_p}$ be the tt-category of global spectra for the family $\cat E_p$, as introduced by Schwede \cite{Schwedebook}, and consider its rationalization $\cat D(\cat E_p) \coloneqq \Sp_{\cat E_p,\bbQ}$. Even though $\cat D(\cat E_p)^{\omega}$ is not rigid, it is \emph{standard} in the sense that every ideal in $\cat D(\cat E_p)^{\omega}$ is radical. In forthcoming joint work with Barrero, Pol, Strickland, and Williamson, we construct a homeomorphism
        \begin{equation}\label{eq:globalreptheory}
            \Spc(\cat D(\cat E_p)^{\omega}) \cong \Spec(\bbZ)_{\inv}. 
        \end{equation}
    In particular, the spectrum of $\cat D(\cat E_p)^{\omega}$ is infinite and inverse-Noetherian. Our proof also shows that every prime ideal in $\cat D(\cat E_p)^{\omega}$ is principal, which alternatively follows abstractly from \cref{prop:fgttcharacterization}. Moreover, \cref{thm:main} implies that $\Spc(\cat D(\cat E_p)^{\omega})$ cannot be weakly Noetherian, which alternatively could be deduced from \cref{eq:globalreptheory}.
\end{Exa}

\subsection*{Acknowledgements}
I would like to thank Scott Balchin for many related discussions on the question to what extent structural properties of a tt-category are reflected in the topology of its spectrum, as well as Paul Balmer, Drew Heard, and Beren Sanders for helpful comments. TB is grateful to the Max Planck Institute for Mathematics, is supported by the ERC under Horizon Europe~(grant~No.~101042990), and would like to thank the Isaac Newton Institute for Mathematical Sciences, Cambridge, for support and hospitality during the programme `Equivariant homotopy theory in context', where work on this paper was undertaken. This work was supported by EPSRC grant EP/Z000580/1.

\bibliographystyle{alpha}
\bibliography{reference}

\newcommand{\etalchar}[1]{$^{#1}$}
\begin{thebibliography}{BCH{\etalchar{+}}25}

\bibitem[ABHS24]{ABHS2024pp}
Gregory {Arone}, Tobias {Barthel}, Drew {Heard}, and Beren {Sanders}.
\newblock {The spectrum of excisive functors}.
\newblock {\em arXiv e-prints, accepted for publication in {Invent.~Math.}}, page arXiv:2402.04244, February 2024.

\bibitem[Bal05]{Balmer2005}
Paul Balmer.
\newblock The spectrum of prime ideals in tensor triangulated categories.
\newblock {\em J. Reine Angew. Math.}, 588:149--168, 2005.

\bibitem[BBB24]{BBB2024pp}
Scott {Balchin}, David {Barnes}, and Tobias {Barthel}.
\newblock {Profinite equivariant spectra and their tensor-triangular geometry}.
\newblock {\em arXiv e-prints}, page arXiv:2401.01878, January 2024.

\bibitem[BBG23]{BalchinBarthelGreenlees2023pp}
Scott {Balchin}, Tobias {Barthel}, and J.~P.~C. {Greenlees}.
\newblock {Prismatic decompositions and rational $G$-spectra}.
\newblock {\em arXiv e-prints}, page arXiv:2311.18808, November 2023.

\bibitem[BBI{\etalchar{+}}23]{BBIKP2023pp}
Tobias {Barthel}, Dave {Benson}, Srikanth~B. {Iyengar}, Henning {Krause}, and Julia {Pevtsova}.
\newblock {Lattices over finite group schemes and stratification}.
\newblock {\em arXiv e-prints}, page arXiv:2307.16271, July 2023.

\bibitem[BCH{\etalchar{+}}25]{BCHNP2025}
Tobias Barthel, Nat\`alia Castellana, Drew Heard, Niko Naumann, and Luca Pol.
\newblock Quillen stratification in equivariant homotopy theory.
\newblock {\em Invent. Math.}, 239(1):219--285, 2025.

\bibitem[BCHS23]{BCHS2023pp}
Tobias {Barthel}, Natalia {Castellana}, Drew {Heard}, and Beren {Sanders}.
\newblock {Cosupport in tensor triangular geometry}.
\newblock {\em arXiv e-prints, accepted for publication in {Ast\'erisque}}, page arXiv:2303.13480, March 2023.

\bibitem[BCHS24]{BCHS2024}
Tobias Barthel, Nat\`alia Castellana, Drew Heard, and Beren Sanders.
\newblock On surjectivity in tensor triangular geometry.
\newblock {\em Math. Z.}, 308(4):Paper No. 65, 7, 2024.

\bibitem[BF11]{BalmerFavi2011}
Paul Balmer and Giordano Favi.
\newblock Generalized tensor idempotents and the telescope conjecture.
\newblock {\em Proc. Lond. Math. Soc. (3)}, 102(6):1161--1185, 2011.

\bibitem[BG22]{BalmerGallauer2022pp}
Paul {Balmer} and Martin {Gallauer}.
\newblock {The tt-geometry of permutation modules. Part I: Stratification}.
\newblock {\em arXiv e-prints}, page arXiv:2210.08311, October 2022.

\bibitem[BG25]{BalmerGallauer2025}
Paul Balmer and Martin Gallauer.
\newblock The spectrum of {A}rtin motives.
\newblock {\em Trans. Amer. Math. Soc.}, 378(3):1733--1754, 2025.

\bibitem[BGH20]{BGH2020}
Tobias Barthel, J.~P.~C. Greenlees, and Markus Hausmann.
\newblock On the {B}almer spectrum for compact {L}ie groups.
\newblock {\em Compos. Math.}, 156(1):39--76, 2020.

\bibitem[BHS23]{BHS2023}
Tobias Barthel, Drew Heard, and Beren Sanders.
\newblock Stratification in tensor triangular geometry with applications to spectral {M}ackey functors.
\newblock {\em Camb. J. Math.}, 11(4):829--915, 2023.

\bibitem[BIK11]{BIK2011}
David~J. Benson, Srikanth~B. Iyengar, and Henning Krause.
\newblock Stratifying modular representations of finite groups.
\newblock {\em Ann. of Math. (2)}, 174(3):1643--1684, 2011.

\bibitem[Coh50]{Cohen1950}
I.~S. Cohen.
\newblock Commutative rings with restricted minimum condition.
\newblock {\em Duke Math. J.}, 17:27--42, 1950.

\bibitem[DST19]{DST2019}
Max Dickmann, Niels Schwartz, and Marcus Tressl.
\newblock {\em Spectral spaces}, volume~35 of {\em New Mathematical Monographs}.
\newblock Cambridge University Press, Cambridge, 2019.

\bibitem[Hoc69]{Hochster}
M.~Hochster.
\newblock Prime ideal structure in commutative rings.
\newblock {\em Trans. Amer. Math. Soc.}, 142:43--60, 1969.

\bibitem[HS98]{HopkinsSmith1998}
Michael~J. Hopkins and Jeffrey~H. Smith.
\newblock Nilpotence and stable homotopy theory. {II}.
\newblock {\em Ann. of Math. (2)}, 148(1):1--49, 1998.

\bibitem[Nee92]{Neeman92}
Amnon Neeman.
\newblock The connection between the {$K$}-theory localization theorem of {T}homason, {T}robaugh and {Y}ao and the smashing subcategories of {B}ousfield and {R}avenel.
\newblock {\em Ann. Sci. \'Ecole Norm. Sup. (4)}, 25(5):547--566, 1992.

\bibitem[{San}17]{Sanders2017pp}
William~T. {Sanders}.
\newblock {Support and vanishing for non-Noetherian rings and tensor triangulated categories}.
\newblock {\em arXiv e-prints}, page arXiv:1710.10199, October 2017.

\bibitem[Sch18]{Schwedebook}
Stefan Schwede.
\newblock {\em Global homotopy theory}, volume~34 of {\em New Mathematical Monographs}.
\newblock Cambridge University Press, Cambridge, 2018.

\bibitem[Tho97]{Thomason1997}
R.~W. Thomason.
\newblock The classification of triangulated subcategories.
\newblock {\em Compositio Math.}, 105(1):1--27, 1997.

\bibitem[{Zou}23]{Zou2023pp}
Changhan {Zou}.
\newblock {Support theories for non-Noetherian tensor triangulated categories}.
\newblock {\em arXiv e-prints}, page arXiv:2312.08596, December 2023.

\end{thebibliography}

\end{document}